\documentclass{amsart}
\usepackage{amssymb,amstext,amsmath,amscd,amsthm,amsfonts,enumerate,graphicx,latexsym}
\usepackage[usenames]{color}
\usepackage[all]{xy}

\def\SCM{{\underline{\mathrm{CM}}}}

\def\Rs{R^{\sharp }}

\def\Rss{R^{\sharp \sharp }}

\def\End{\mathrm{End}}

\def\Ker{\operatorname{Ker}}
\def\Im{\operatorname{Im}}

\def\tr{\mathrm{tr}}

\def\depth{\mathrm{depth}}

\def\m{\mathfrak m}
\def\p{\mathfrak p}

\def\q{\mathfrak q}
\def\r{\mathfrak r}

\def\Fitt{\mathcal F} 
\newtheorem{theorem}{Theorem}[section]
\newtheorem{lemma}[theorem]{Lemma}
\newtheorem{corollary}[theorem]{Corollary}
\newtheorem{proposition}[theorem]{Proposition}
\theoremstyle{definition}
\newtheorem{definition}[theorem]{Definition}

\newtheorem{example}[theorem]{Example}
\newtheorem{remark}[theorem]{Remark}
\newtheorem*{claim}{Claim}
\numberwithin{equation}{section}

\begin{document}
\allowdisplaybreaks
\title[Degenerations of modules over commutative rings]{Degenerations over $(A_\infty)$-singularities and construction of degenerations over commutative rings}
\author{Naoya Hiramatsu} 
\address{Department of general education, National Institute of Technology, Kure College, 2-2-11, Agaminami, Kure Hiroshima, 737-8506 Japan}
\email{hiramatsu@kure-nct.ac.jp}
\author{Ryo Takahashi}
\address{Graduate School of Mathematics, Nagoya University, Furocho, Chikusaku, Nagoya, Aichi 464-8602, Japan}
\email{takahashi@math.nagoya-u.ac.jp}
\urladdr{http://www.math.nagoya-u.ac.jp/~takahashi/}
\author{Yuji Yoshino}
\address{Department of Mathematics, Graduate School of Natural Science and Technology, Okayama University, Okayama 700-8530, Japan}
\email{yoshino@math.okayama-u.ac.jp}
\date{\today}
\thanks{NH was supported by JSPS Grant-in-Aid for Young Scientists 15K17527.
RT was supported by JSPS Grant-in-Aid for Scientific Research 16K05098.
YY was supported by JSPS Grant-in-Aid for Scientific Research 26287008.}
\begin{abstract}
We give a necessary condition of degeneration via matrix representations, and consider degenerations of indecomposable Cohen-Macaulay modules over hypersurface singularities of type ($A_\infty$).
We also provide a method to construct degenerations of finitely generated modules over commutative rings.
\end{abstract}
\maketitle
\section{Introduction}

The notion of degenerations of modules has been a central subject in the representation theory of algebras.
Let $k$ be a field, and let $R$ be a $k$-algebra.
Let $M,N$ be $R$-modules with the same finite dimension as $k$-vector spaces.
We say that {\em $M$ degenerates to $N$} if in the module variety the point corresponding to $M$ belongs to the Zariski closure of the orbit of the point corresponding to $N$.
This definition is only available for modules of finite length.
The third author \cite{Y04} thus introduces a scheme-theoretic definition of degenerations for arbitrary finitely generated modules, and extends the characterization of degenerations due to Riedtmann \cite{R86} and Zwara \cite{Z00}.
Interest of this subject is to describe relations of degenerations of modules \cite{B96, R86, Y02, Y04, Z99, Z00}.
In \cite{HY13}, the first and third authors give a complete description of degenerations over a ring of even-dimensional simple hypersurface singularity of type ($A_n$).

The first purpose of this paper is to give a necessary condition for the degenerations of maximal Cohen-Macaulay modules by considering it via matrix representations over the base regular local ring. 
As an application, we give the description of degenerations of indecomposable Cohen-Macaulay modules over hypersurface singularities of type ($A_\infty$).

\begin{theorem}\label{a}
Let $k$ be an algebraically closed field.
\begin{enumerate}[\rm(1)]
\item
Let $R=k[\![x,y]\!]/(x^2)$ be a hypersurface of dimension one.
Then $(x, y^a)R$ degenerates to $(x, y^b)R$ if and only if $a \leq b$ and $a \equiv b \mod 2$.
\item
Let $R=k[\![x,y,z]\!]/(xy)$ be a hypersurface of dimension two.
Then for all $a<b$, $(x, z^a)R$ and $(y, z^a)R$ do not degenerate to $(x, z^b)R$ and $(y, z^b)R$, respectively.
\end{enumerate}
\end{theorem}

\noindent
The proof of this theorem requires the fact that the Bass-Quillen conjecture holds for a formal power series ring over a field (\cite[Chapter V Theorem 5.1]{Lam}), which is the only place where the assumption that $R$ is complete and equicharacteristic is necessary.
Our proof also requires the base field to be algebraically closed.
Note that a complete equicharacteristic local hypersurface $R$ has countable Cohen-Macaulay representation type if and only if $R$ has either of the two forms in Theorem \ref{a}.

The second purpose of this paper is to provide a method to construct degenerations.
We obtain the following theorem.

\begin{theorem}\label{b}
Let $R$ be a commutative noetherian algebra over a field, and let $L$ be a finitely generated $R$-module.
Let $\alpha$ be an endomorphism of $L$ with $\Im\alpha=\Ker\alpha$, and let $x\in R$ be an $L$-regular element.
Then the following statements hold.
\begin{enumerate}[\rm(1)]
\item
Any submodule $M$ of $L$ containing $\alpha(L)+xL$ degenerates to $N:=\alpha(M)+x^2L$.
\item
If $L,M$ are Cohen-Macaulay, then so is $N$.
\end{enumerate}
\end{theorem}

\noindent
This theorem is actually a corollary of our main result in this direction, which holds over arbitrary commutative rings and is proved by elementary calculations of matrices; we do not use matrix representations.
We also give various applications, one of which gives an alternative proof of the `if' part of Theorem \ref{a}(1).

The organization of this paper is as follows.
In Section \ref{matrix representation}, we review the notion of degenerations of Cohen-Macaulay modules by means of matrix representations, and give a necessary condition for degenerations (Corollary \ref{fitting}).
Using this, in Section \ref{ainf}, we describe degenerations over hypersurface singularities of type ($A_{\infty} $) with dimension at most two (Theorems \ref{odd dim} and \ref{even dim}).
In Section \ref{construction}, we give several ways to construct degenerations (Theorem \ref{1}), and provide applications.
In Section \ref{Knorrer} we attempt to extend Theorem \ref{a} to higher dimension (Example \ref{14}).

\section{Matrix representation}\label{matrix representation} 

We begin with recalling the definition of degenerations of finitely generated modules. For details, we refer the reader to \cite{Y02, Y04, Y11}.  

\begin{definition}\label{degeneration}
Let $R$ be a noetherian algebra over a field $k$.
Let $M,N$ be finitely generated left $R$-modules. 
We say that {\em $M$ degenerates to $N$} (or {\em $N$ is a degeneration of $M$}) if there exists a discrete valuation $k$-algebra $(V, tV, k)$ and a finitely generated left $R\otimes _{k} V$-module $Q$ such that $Q$ is flat as a $V$-module, $Q/tQ \cong N$ as an $R$-module, and $Q_t\cong M\otimes _{k} V_t$ as an $R\otimes _{k} V_t$-module.
When this is the case, we say that {\em $M$ degenerates to $N$ along $V$}.
\end{definition}

\begin{remark}\label{remark of degenerations 1}
Let $R$ be a noetherian algebra over a field $k$ and let $M$, $N$ and $L$ be finitely generated left $R$-modules. \\
(1) $M$ degenerates to $N$ if and only if there exists a short exact sequence of finitely generated left $R$-modules 
$$
\begin{CD}
0 @>>> Z @>{\left ( \begin{smallmatrix}h \\ g \end{smallmatrix}\right)}>> M \oplus Z @>>> N @>>> 0, 
\end{CD}
$$ 
where the endomorphism $g$ of $Z$ is nilpotent. 
See \cite[Theorem 1]{Z00} and \cite[Theorem 2.2]{Y04}. 
Note that if $M$ and $N$ is Cohen-Macaulay, then so is $Z$ (see \cite[Remark 4.3]{Y04}). \\
(2) Suppose that there is an exact sequence  
$0 \to L \to M \to N \to 0$.  
Then $M$ degenerates to $L \oplus N$. 
See \cite[Remark 2.5]{Y04} for the detail. \\
(3) Assume that $R$ is commutative and suppose that $M$ degenerates to $N$. 
Then the $i$th Fitting ideal of $M$ contains that of $N$ for all $i \geq 0$. 
Namely, denoting the $i$th Fitting ideal of an $R$-module $M$ by $\Fitt_{i}^{R} (M)$, we have $\Fitt_{i}^{R} (M) \supseteq \Fitt_{i}^{R} (N)$ for all $i \geqq 0$. 
See \cite[Theorem 2.5]{Y11}. 
\end{remark}

\begin{remark}\label{remark of degenerations 2}
We can always take the localization $k[t]_{(t)}$ of a polynomial ring as a discrete valuation ring $V$ as in Definition \ref{degeneration}. 
Moreover, let $T= k[t] \backslash (t)$ and $T'= k[t] \backslash \{0 \}$. 
Then we also have $R\otimes _k V = T^{-1} R[t]$ and $R\otimes _k V_t = T' {}^{-1} R[t]$. 
See \cite[Corollary 2.4]{Y04}. 
\end{remark}

Now let us recall the definition of a matrix representation.

\begin{definition}
Let $R$ be a commutative noetherian local ring.
Let $S$ be a Noether normalization of $R$, that is, $S$ is a subring of $R$ which is a regular local ring such that $R$ is a finitely generated $S$-module.
Then $R$ is Cohen-Macaulay if and only if $R$ is $S$-free, and a finitely generated $R$-module is (maximal) Cohen-Macaulay\footnote{Throughout this paper, ``Cohen-Macaulay" means ``maximal Cohen-Macaulay".} if and only if $M$ is $S$-free.

Let $R,M$ be Cohen-Macaulay.
Then $M\cong S^n$ for some $n \geq 0$, and we have a $k$-algebra homomorphism 
$$
R \to \End _{S}(M) \cong  S^{n \times n}.
$$
This is called a {\em matrix representation} of $M$ over $S$. 
\end{definition}

\begin{example}\label{1-dim}
(1) Let $S$ be a regular local ring and $R = S[\![x]\!]/(f)$ where $f= x^n + a_1 x^{n-1} + \cdots + a_n$ with $a_i \in S$.  
Then $R$ is a Cohen-Macaulay ring with an $S$-basis $\{ 1, x, x^2, ..., x^{n-1} \}$ and each matrix representation of Cohen-Macaulay $R$-module $M$ comes from the action of $x$ on $M$. \\
(2) 
Let $R = k[\![x, y]\!]/(x^2 )$ with a field $k$ of characteristic not two. 
It is known that the isomorphism classes of indecomposable Cohen-Macaulay $R$-modules are the following:
$$
R, \quad R/ (x), \quad (x, y^n) \quad n \geq 1.
$$
Giving the matrix representations over $k[\![y]\!]$ of these modules is equivalent to giving the following square-zero matrices; see also \cite[(6.5)]{Y}.
$$
\left(\begin{smallmatrix}  0 & 1 \\ 0 & 0 \end{smallmatrix}\right), \quad(0), \quad \left(\begin{smallmatrix}  0 & y^{n} \\ 0 & 0 \end{smallmatrix}\right) \quad n \geq 1. 
$$
\end{example}

\begin{proposition}\label{free} 
Let $(R,\m,k)$ be a commutative noetherian complete equicharacteristic local ring.
Then one can take a formal power series ring $S$ over $k$ as a Noether normalization of $R$.
Assume that $k$ is algebraically closed.
Let $M,N$ be Cohen-Macaulay $R$-modules such that $M$ degenerates to $N$ along $V$.
Let $Q$ be a finitely generated $R\otimes _k V$-module which gives the degeneration. 
Then $Q$ is free as an $S \otimes _k V$-module. 
\end{proposition}

\begin{proof} 
Since $V = T^{-1} k[t]$ where $T= k[t] \backslash (t)$ by Remark \ref{remark of degenerations 2}, we can take a finitely generated $R[t]$-submodule $Q'$ of $Q$ such that $T^{-1} Q' = Q$. 
Then $Q'$ is flat over $k[t]$, $Q'_0 \cong N$ and $Q' _{c} \cong M$ for each $c \in k^{\times}$. Here $Q' _c$ is defined to be $Q' / (t-c)Q'$ for an element $c \in k$; see \cite[Theorem 3.2]{Y04}. 
We show that $Q'$ is free as an $S[t]$-module. 
Set $T' = k[t] \backslash \{ 0 \}$ and $\p$ be a prime ideal of $S[t]$. 

Suppose $T' \cap \p \not= \emptyset$. Since $k$ is an algebraically closed field, there is $c \in k$ such that $t-c \in \p$. 
Then, 
$$
Q'_{\p}/(t-c)Q'_{\p} \cong (Q' / (t-c)Q')_{\over{\p}}. 
$$ 
By the choice of $Q'$, the right-hand side is isomorphic to $N_{\over{\p}}$ or $M_{\over{\p}}$, so that $S_{\over{\p}}$-free. 
Thus $Q'_{\p}/(t-c)Q'_{\p}$ is free as an $S_{\over{\p}} = (S[t]/(t-c))_{\over{\p}}$-module. 
Thus, $Q'_{\p}$ is a Cohen-Macaulay module over the regular local ring $S[t]_{\p}$. 
Hence, $Q'_{\p}$ is free as an $S[t]_{\p}$-module. 

Suppose $T' \cap \p = \emptyset $. Then
$$
Q'_{\p} \cong (T' {}^{-1} Q')_{\p '} \cong  (T' {}^{-1} (T^{-1} Q')_{\p '} \cong (T' {}^{-1} Q)_{\p '} \cong (M \otimes _{k} V_t)_{\p '}. 
$$
Hence $Q'_{\p}$ is a free $S[t]_{\p}$-module.

Therefore $Q'$ is a projective $S[t]$-module. 
By the fact on the Bass-Quillen conjecture \cite[Chapter V Theorem 5.1]{Lam}, each projective $S[t]$-module is $S[t]$-free since $S$ is a formal power series ring. 
Hence $Q'$ is free as an $S[t]$-module, so that $Q = T^{-1} Q'$ is free as a $T^{-1}S[t] = S \otimes _k V$-module.  
\end{proof}

The above proposition enables us to consider the matrix representation of $Q$ over $S \otimes _k V$ because it is free.
For matrix representations $\mu,\nu : R \to S^{n \times n}$, we write $\mu \cong \nu$ if there exists an invertible matrix $\alpha$ such that $\alpha^{-1} \mu (r) \alpha = \nu (r)$ for each $r \in R$. 
Such a matrix $\alpha$ is called an {\em interwining matrix} of $\mu$ and $\nu$. 

\begin{corollary}\label{degeneration via MR}
Let $R,k,S,V$ be as in Proposition \ref{free}. Let $M$ and $N$ be Cohen-Macaulay $R$-modules and $\mu$ and $\nu$ the matrix representations of $M$ and $N$ over $S$ respectively. 
Then $M$ degenerates to $N$ along $V$ if and only if there exists a matrix representation $\xi$ over $S \otimes _k V$ such that $\xi \otimes _V V/tV \cong \nu$ and $\xi \otimes _V V_{t} \cong \mu \otimes _k V_{t}$  hold.
\end{corollary}

Under the situation of the proposition above, $R$ is a finitely generated $S$-module. 
Set $u_1, ..., u_t$ be generators of $R$ as an $S$-module. 
Then giving a matrix representation $\mu$ over $S$ is equivalent to giving a tuple of matrices $(\mu_1, ..., \mu_t)$, where $\mu_i = \mu (u_i)$ for $i= 1, ..., t$.

\begin{corollary}\label{fitting}
Let $R,k,S,V$ be as in Proposition \ref{free} and let $M$ be a Cohen-Macaulay $R$-module. 
Let $Q$ be a free $S \otimes _k V$-module and suppose that $Q_{t}$ is isomorphic to $M \otimes _{k} V_{t}$. 
We denote by $(\mu_1, \dots, \mu_t)$ (resp. $(\xi_1, \dots, \xi_t)$) the matrix representation of $M$ (resp. $Q$) over $S$ (resp. $S \otimes_k V$). 
Then we have the following equalities in $S \otimes _k V$ for $i= 1, \dots, t$:  

\begin{enumerate}[\rm(1)]
\item 
$\tr (\xi_i ) =  \tr (\mu_i )$ and $\det (\xi_i ) = \det (\mu_i )$,   
\item 
For all $j \geq 0$, there exist $l$, $l'$ such that $I_j (\xi_i ) = t^{l} I_j (\mu_i )$ and $t^{l'} I_j (\xi_i ) =  I_j (\mu_i )$.
\end{enumerate}
Here, $I_j (-)$ stands for the ideal generated by the $j$-minors.
\end{corollary}

\begin{proof}
By the assumption, there exists an invertible matrix $\alpha $ with entries in $S \otimes _k V_t$ such that $\alpha ^{-1} \xi_i \alpha = \mu_i$. 
Thus  $\tr (\xi_i ) = \tr (\mu_i )$ and $\det (\xi_i ) = \det (\mu_i )$. 
Moreover we have $I_j (\xi _i ) = I_j (\mu _i )$ in $S \otimes _k V_t$. 
Hence we have the equalities in the lemma. 
\end{proof} 

Taking Corollary \ref{degeneration via MR} into account, Corollary \ref{fitting} gives a necessary condition of the degeneration. 
In the next section, we apply the condition to $(A _{\infty} ^d)$-singularities of dimension at most two. 

\section{Degenerations of Cohen-Macaulay modules over $(A _{\infty} )$-singularities}\label{ainf}

Recall that a Cohen-Macaulay local ring $R$ is said to have {\em countable Cohen-Macaulay representation type} if there exist infinitely but only countably many isomorphism classes of indecomposable Cohen-Macaulay $R$-modules.
Let $R$ be a complete equicharacteristic local hypersurface with residue field $k$ of characteristic not two.
Then $R$ has countable Cohen-Macaulay representation type if and only if $R$ is isomorphic to the ring $k[\![x_0 , x_1 , x_2 , \cdots, x_d]\!]/(f)$, where $f$ is either of the following:
$$
f=\begin{cases}
x_1^2+\cdots+x_d^2 & (A_\infty),\\
x_0^2x_1+x_2^2+\cdots+x_d^2 & (D_\infty).
\end{cases}
$$
Moreover, when this is the case, the indecomposable Cohen-Macaulay $R$-modules are classified completely; we refer the reader to \cite[\S1]{AIT12} for more information.

In this section, we describe the degenerations of indecomposable Cohen-Macaulay $R$-modules in the case where $R$ has type $(A_\infty)$ and $\dim R=1,2$.
Throughout the rest of this section, we assume that $k$ is an algebraically closed field of characteristic not two.
Let us start by the $1$-dimensional case.

\begin{theorem}\label{odd dim}
Let $R = k[\![x, y]\!]/(x^2 )$.  
Then $(x, y^a)$ degenerates to $(x, y^b)$ if and only if $a \leq b$ and $a \equiv b \mod 2$. 
\end{theorem}

\begin{proof}
First we notice that $a \leq b$ if $(x, y^a)$ degenerates to $(x, y^b)$ by Remark \ref{remark of degenerations 1}(3).
As mentioned in Example \ref{1-dim}(2), the matrix representations of $(x, y^n)$ for $n \geq 0$ are 
$
\left( \begin{smallmatrix}  0 & y^{n} \\ 0 & 0 \end{smallmatrix} 
\right)$. 
Here $(x, y^0) = R$. 
Suppose that $a \equiv b \mod 2$. 
We consider $R \otimes _k V$-module $Q$ whose matrix representation $\xi$ over $S \otimes _k V$ is 
$$\xi = 
\left( \begin{smallmatrix}
ty^{\frac{a + b}{2}} & y^b \\
-t^2 y^a &- ty^{\frac{a + b}{2}} \\
\end{smallmatrix}\right).
$$
Then one can show that $Q$ gives the degeneration. 
 
To show the converse, we prove the following claim. 

\begin{claim}
$\left ( \begin{smallmatrix}  0 & 1 \\ 0 & 0  \end{smallmatrix} \right )$ never degenerates to $\left ( \begin{smallmatrix}  0 & y^{2m+1} \\ 0 & 0  \end{smallmatrix} \right )$, that is, $R$ never degenerates to $(x, y^{2m+1})$.
\end{claim}

\begin{proof}[Proof of Claim]
After applying elementary row and column operations, we may assume that the matrix representation of $Q$ over $S\otimes _k V$ which gives the degeneration is of the form:
$$
\xi = \left( \begin{smallmatrix}  0 & y^{2m+1} \\ 0 & 0 \end{smallmatrix}\right) + t \left( \begin{smallmatrix}  \alpha & \beta \\ \gamma & \delta \end{smallmatrix}\right) = \left( \begin{smallmatrix}  t \alpha & y^{2m+1} + t\beta \\ t\gamma & t\delta \end{smallmatrix}\right ).
$$
Then $\xi \otimes V_t \cong \mu \otimes V_t$. 
By Corollary \ref{fitting}, we have $t \delta = -t \alpha$ since $\tr (\xi ) = tr (\mu ) = 0$. 
Moreover,  
\begin{eqnarray}
\det \xi = -t^2 \alpha ^2 - t \gamma (y^{2m+1} + t \beta ) = 0, \\ 
I_1 (\xi ) = (t\alpha , t \gamma , y^{2m+1} + t \beta) \supseteq (t^l ) \text{ for some } l. \nonumber 
\end{eqnarray}
Note that the above equalities are obtained in $S\otimes _k V$.  
From the equation (3.1), we have
$
t \alpha ^2 = \gamma (y^{2m +1} + t \beta ).  
$ 
Since $t$ does not divide $y^{2m +1} + t \beta $, $t$ divides $\gamma$, so that $\gamma = t \gamma '$ for some $\gamma ' \in S\otimes _k V$. 
Hence we also have the equality in $S\otimes _k V$:  
\begin{equation}\label{odd}
 \alpha ^2 = \gamma ' (y^{2m +1} + t \beta ). 
\end{equation}
Since $S$ is factorial, so is $S[t]$. 
Thus $S\otimes _k V = T^{-1} S[t] $ is also factorial. 
Take the unique factorization into prime elements of $y^{2m + 1} + t \beta$: 
$$
y^{2m +1} + t \beta = P^{e_1} _{1}P^{e_2} _{2} \cdots P^{e_n} _{n}. 
$$  
Then, there exists $i$ such that $e_i$ is an odd number. 
Since the equation (\ref{odd}) holds, $P_i$ divides $\alpha$, so that $P_i$ also divides $\gamma '$. 
Therefore, 
$
(P_i ) \supseteq I_i (\xi ) \supseteq (t ^l), 
$ 
so that $P_i = t$. 
This makes contradiction since $t$ cannot divide $y^{2m +1} + t \beta$.
\renewcommand{\qedsymbol}{$\diamondsuit$}
\end{proof}

Now we assume that $\left ( \begin{smallmatrix}  0 & y^{a} \\ 0 & 0  \end{smallmatrix} \right )$ degenerates to $\left ( \begin{smallmatrix}  0 & y^{b} \\ 0 & 0  \end{smallmatrix} \right )$. 
Then we also have the following matrix representation of $Q$ over $S\otimes _k V$:
$$
\xi = \left( \begin{smallmatrix}  0 & y^{b} \\ 0 & 0 \end{smallmatrix} \right)+ t \left( \begin{smallmatrix}  \alpha & \beta \\ \gamma & -\alpha \end{smallmatrix} \right) = \left( \begin{smallmatrix}  t \alpha & y^{b} + t \beta \\ t\gamma & -t\alpha \end{smallmatrix}\right).
$$
Then we also have $t^2 \alpha ^2  =  t \gamma (y^{b} + t \beta )$, $(t\alpha , t \gamma , y^{b} + t \beta) \supseteq t^l (y^a )$ for some $l$, and
\begin{eqnarray}
t^{l'} (t\alpha , t \gamma , y^{b} + t \beta) \subseteq (y^a ) \text{ for some } l'.  
\end{eqnarray}
From (3.3), we see that $y^a$ divides $t\alpha$, $\ t \gamma$ and $\ y^b + t \beta$. 
This yields that 
$$
t^2 {\alpha '} ^2  =  t \gamma ' (y^{b-a} + t \beta' )\text{ and }
(t\alpha ' , t \gamma '  , y^{b-a} + t \beta') \supseteq (t^l ) \text{ for some } l,
$$ 
where $\alpha = y^a \alpha'$, $\gamma = y^a \gamma'$and $y^{b} + t \beta = y^a (y^{b-a} + t \beta')$. 
If $a \not \equiv b \mod 2$, $b-a$ is an odd number. 
As in the proof of the claim above, this never happen. 
Therefore $a  \equiv b \mod 2$. 
\end{proof}

Next we consider the $2$-dimensional case, that is, let $R= k[\![x, y, z]\!]/(xy)$. 
The nonisomorphic indecomposable Cohen-Macaulay $R$-modules are $R$ and the following ideals; see \cite[Proposition 2.2]{AIT12}. 
$$
(x), \quad (y), \quad (x, z^n), \quad (y, z^n) \quad n \geq 1.  
$$

\begin{theorem}\label{even dim}
Let $R = k[\![x, y, z]\!]/(xy)$. 
Then $(x, z^a)$ (resp. $(y, z^a)$) never degenerates to $(x, z^b)$ and $(y, z^b)$ for all $a < b$. 
\end{theorem}

\begin{proof}
Replacing $Y$ with $x - y$, we may consider $k[\![x, Y, z]\!]/(x^2-Yx)$ as $R$. 
We consider the ideals $(x, z^a)$ and $(x-Y, z^a)$ instead of the ones in the assertion.  
We rewrite $Y$ by $y$.
The matrix representations of $(x, z^a)$ and $(x-y, z^a)$ over $S=k[\![y, z]\!]$ are $\left ( \begin{smallmatrix}  y & z^{a} \\ 0 & 0 \end{smallmatrix}\right )$ and $\left ( \begin{smallmatrix}  0 & z^{a} \\ 0 & y \end{smallmatrix}\right)$ respectively. 
See Remark \ref{ex of 2-dim} below.

Suppose that $\mu = \left ( \begin{smallmatrix}  y & z^a \\ 0 & 0 \end{smallmatrix}\right)$ degenerates to $\left ( \begin{smallmatrix}  y & z^b \\ 0 & 0 \end{smallmatrix}\right )$ along $V$, and let $Q$ be a finitely generated $R\otimes _k V$-module which gives the degeneration.
By elementary row and column operations, we may assume that the matrix representation $\xi$ of $Q$ over $S\otimes _k V$ which gives the degeneration is of the form:
$$
\xi = \left( \begin{smallmatrix}  y & z^{b} \\ 0 & 0 \end{smallmatrix}\right) + t \left(\begin{smallmatrix}  \alpha & \beta \\ \gamma & \delta \end{smallmatrix} \right) = \left( \begin{smallmatrix}  y + t \alpha & z^{b} + t\beta \\ t\gamma & t\delta \end{smallmatrix}\right).
$$
Since $\xi \otimes V_t \cong \mu \otimes V_t$, by Corollary \ref{fitting}, we have
\begin{eqnarray}
\tr (\xi ) = y, \\  
\det ( \xi )= (y + t \alpha )t \delta  - t \gamma (z^{b} + t \beta ) = 0, \\ 
I_1 (\xi ) = (y + t\alpha , t \gamma , z^{b} + t \beta,  t \delta) \supseteq t^l (z^a, y) \text{ for some } l. 
\end{eqnarray}
By (3.4) and (3.5), we have 
\begin{eqnarray}
-\alpha (y + t \alpha)  =  \gamma (z^{b} + t \beta ). 
\end{eqnarray}
Set $P=z^b +t\beta$. 
Then, since $P \equiv z^{b}$ modulo $t$, $P$ divides $\alpha$. 
Consider the equation (3.7) in $S\otimes _k V/ y(S\otimes _k V)$: 
$$
-t \overline{\alpha}\cdot  \overline{\alpha}  =  \overline{\gamma } \cdot \overline{z^{b} + t \beta }. 
$$
Since $P$ divides $\alpha$, $\overline{P}$ divides $\overline{\alpha }$. 
Thus $\overline{P}$ also divides $\overline{\gamma }$. 
This yields that 
$
I_1 (\xi ) = (y + t\alpha , t \gamma , z^{b} + t \beta, - t \alpha) \subseteq (P, y).
$ 
By (3.6), 
$
t^l (z^a, y) \subseteq (P, y). 
$
Let $\p$ be a minimal prime ideal of $(P, y)$. 
Notice that $\p$ has height 2 and contains $t^l z$ and $y$. 
Suppose that $\p$ contains $t$. 
Then $\p$ also contains $z$ since $\p$ contains $P = z^b + t \beta$, so that $\p = ( z, y, t)$. 
This makes contradiction.  
Hence $\p = (z, y)$.   
Consider the above inclusion of ideals in $S\otimes _k V/ y(S\otimes _k V)$ again, and then we see that $\overline{(P, y)} = (\overline{P} ) \ni \overline{t^l z^a}$. 
Since $ (\overline{P} )$ are also contained in $\overline{\p} = (\overline{z})$, $ \overline{P} = \overline{P' z}$ for some $\overline{P'} \in  S\otimes _k V/ y(S\otimes _k V)$.  
Thus $(\overline{P} ) = (\overline{P' z}) \ni \overline{t^l z^a}$, which implies that $(\overline{P' } ) \ni \overline{t^l z^{a-1}}$ as an ideal in $S\otimes _k V/ y(S\otimes _k V)$. 
Note that $P' = z^{b-1} + t P''$ in $S\otimes _k V$ and the minimal prime ideal of $(\overline{P'})$ in $S\otimes _k V/ y(S\otimes _k V)$ is also $(\overline{z})$.  
Hence we can repeat the procedure and we obtain the inclusion of ideals in $S\otimes _k V/ y(S\otimes _k V)$:
$$
( \overline{z^{b-a} + t P'''}) \supseteq (\overline{t^l }).  
$$
However this never happen if $b-a>0$. 
Therefore we obtain the conclusion. 
The same proof is valid for showing the remaining claim that $(x, z^a)$ (resp. $(y, z^a)$ ) does not degenerate to $(y, z^b)$ (resp. $(x, z^b)$ or $(y, z^b)$).
\end{proof}

\begin{remark}\label{ex of 2-dim}
Let $R = k[\![x, y, z]\!]/(x^2-xy )$ and let $M=(x,z^n)$, $N=(x-y,z^n)$ be ideals. 
Set $S=k[\![y, z]\!]$. 
Note that the matrix representations of $M$ and $N$ over $S$ are obtained from the action of $x$ on $M$ and $N$ respectively. 
Since $M$ has a basis $x$ and $z^n$ as an $S$-module, that is, $M \cong xS \oplus z^n S$, the multiplication map $a_x$ by $x$ on $M$ induces the correspondence: 
$$
a_x : xS \oplus z^n S \to xS \oplus z^n S ; \quad \left( \begin{smallmatrix} ax \\ b z^n \end{smallmatrix}\right)   \mapsto \left( \begin{smallmatrix} ayx + bz^n x  \\ 0 \end{smallmatrix}\right). 
$$
Hence the matrix representation of $M$ over $S$ is $\left ( \begin{smallmatrix}  y & z^{n} \\ 0 & 0  \end{smallmatrix} \right )$.
Similarly, since $N \cong (x-y) S \oplus z^n S$ as an $S$-module, the multiplication map $b_x$ on $N$ induces
$$
b_x : (x-y)S \oplus z^n S \to (x-y) S \oplus z^n S ; \quad \left( \begin{smallmatrix} a(x-y) \\ b z^n \end{smallmatrix}\right)   \mapsto \left( \begin{smallmatrix} bz^n(x-y)  \\ byz^n \end{smallmatrix}\right),  
$$
so that the matrix representation of $N$ is $\left ( \begin{smallmatrix}  0 & z^{n} \\ 0 & y  \end{smallmatrix} \right )$.
\end{remark}

\begin{remark}\label{A type}
Araya, Iima and the second author \cite{AIT12} show that the Cohen-Macaulay modules which appear in Theorem \ref{odd dim} (resp. Theorem  \ref{even dim}) are obtained from the extension by $R/(x)$ and itself (resp. $R/(x)$ and $R/(y)$). 
In the case, we have the degeneration by Remark \ref{remark of degenerations 1}(2). 
Summing up this fact, we obtain complete description of the degenerations of indecomposable Cohen-Macaulay modules over $k[\![x, y]\!]/(x^2)$ and $k[\![x, y, z]\!]/(xy)$. 
\end{remark}

\section{Construction of degenerations of modules}\label{construction}

In this section we provide a construction of a nontrivial degeneration over an arbitrary commutative noetherian algebra over a field (without matrix representations).
We begin with establishing a lemma.

\begin{lemma}\label{7}
Let $R$ be a commutative ring, and let $\cdots\xrightarrow{\alpha}L\xrightarrow{\beta}L\xrightarrow{\alpha}L\xrightarrow{\beta}\cdots$ be an exact sequence of $R$-modules.
Let $x\in R$, and set $Z=\alpha(L)+xL$ and $W=\beta(L)+xL$.
\begin{enumerate}[\rm(1)]
\item
There is an exact sequence
$$
\xymatrix@R-1pc@C-1pc{\cdots\ar[rd]\ar[rr]^{\left(\begin{smallmatrix}\alpha&x\\
0&\beta\end{smallmatrix}\right)} && L\oplus L\ar[rr]^{\left(\begin{smallmatrix}\beta&-x\\
0&\alpha\end{smallmatrix}\right)}\ar[rd] && L\oplus L\ar[rd]\ar[rr]^{\left(\begin{smallmatrix}\alpha&x\\
0&\beta\end{smallmatrix}\right)} && L\oplus L\ar[rr]^{\left(\begin{smallmatrix}\beta&-x\\
0&\alpha\end{smallmatrix}\right)}\ar[rd] && \cdots\\
& Z\ar[ru] && W\ar[ru] && Z\ar[ru] && W\ar[ru]}.
$$
\item
Suppose that $x$ is $L$-regular.
Then there are exact sequences:
\begin{align*}
&L\oplus L\xrightarrow{\left(\begin{smallmatrix}\beta&-x\\
0&\alpha
\end{smallmatrix}\right)} L\oplus L \xrightarrow{\left(\begin{smallmatrix}\alpha&x\end{smallmatrix}\right)} L\to L/Z\to0,\\
&L\oplus L \xrightarrow{\left(\begin{smallmatrix}\alpha&x\\
0&\beta\end{smallmatrix}\right)} L\oplus L\xrightarrow{\left(\begin{smallmatrix}\beta&-x\end{smallmatrix}\right)} L\to L/W\to0.
\end{align*}
\end{enumerate}
\end{lemma}

\begin{proof}
(1) Clearly, $\left(\begin{smallmatrix}\alpha&x\\
0&\beta\end{smallmatrix}\right)
\left(\begin{smallmatrix}\beta&-x\\
0&\alpha
\end{smallmatrix}\right)=0$.
Let $s,t\in L$ with $\left(\begin{smallmatrix}\alpha&x\\
0&\beta\end{smallmatrix}\right)\left(\begin{smallmatrix}s\\
t\end{smallmatrix}\right)=\left(\begin{smallmatrix}0\\
0\end{smallmatrix}\right)$.
Then $\alpha(s)+xt=0=\beta(t)$, and $t\in\ker\beta=\Im\alpha$.
We have $t=\alpha(u)$ for some $u\in L$, and $0=\alpha(s)+xt=\alpha(s+xu)$, which implies $s+xu\in\ker\alpha=\Im\beta$.
Writing $s+xu=\beta(v)$ with $v\in L$, we get $\left(\begin{smallmatrix}s\\
t\end{smallmatrix}\right)=\left(\begin{smallmatrix}\beta(v)-xu\\
\alpha(u)\end{smallmatrix}\right)=\left(\begin{smallmatrix}\beta&-x\\
0&\alpha
\end{smallmatrix}\right)\left(\begin{smallmatrix}v\\
u\end{smallmatrix}\right)$.
Thus $\Im\left(\begin{smallmatrix}\beta&-x\\
0&\alpha
\end{smallmatrix}\right)=\ker\left(\begin{smallmatrix}\alpha&x\\
0&\beta\end{smallmatrix}\right)$.
A symmetric argument shows $\Im \left(\begin{smallmatrix}\alpha&x\\
0&\beta\end{smallmatrix}\right)=\ker\left(\begin{smallmatrix}\beta&-x\\
0&\alpha
\end{smallmatrix}\right)$.

(2) Note that $\left(\begin{smallmatrix}
\alpha&x\end{smallmatrix}\right)$ is a submatrix of $\left(\begin{smallmatrix}\alpha&x\\
0&\beta\end{smallmatrix}\right)$; we have $\left(\begin{smallmatrix}\alpha&x\end{smallmatrix}\right)\left(\begin{smallmatrix}\beta&-x\\
0&\alpha
\end{smallmatrix}\right)=0$.
Let $s,t\in L$ with $\left(\begin{smallmatrix}\alpha&x\end{smallmatrix}\right)\left(\begin{smallmatrix}s\\
t\end{smallmatrix}\right)=\left(\begin{smallmatrix}0\\
0\end{smallmatrix}\right)$.
Then $\alpha(s)+xt=0$, and $0=\beta(\alpha(s)+xt)=x\beta(t)$.
Since $x$ is $L$-regular, $\beta(t)=0$.
The above argument shows $\left(\begin{smallmatrix}s\\
t\end{smallmatrix}\right)\in\Im \left(\begin{smallmatrix}\beta&-x\\
0&\alpha
\end{smallmatrix}\right)$.
Thus $\Im \left(\begin{smallmatrix}\beta&-x\\
0&\alpha
\end{smallmatrix}\right)=\ker\left(\begin{smallmatrix}
\alpha&x\end{smallmatrix}\right)$.
A symmetric argument shows $\Im \left(\begin{smallmatrix}\alpha&x\\
0&\beta\end{smallmatrix}\right)=\ker\left(\begin{smallmatrix}
\beta&-x\end{smallmatrix}\right)$.
\end{proof}

Recall that a finitely generated module $M$ over a commutative noetherian ring $R$ is Cohen-Macaulay if $\depth_{R_\p}M_\p\ge\dim R_\p$ for all prime ideals $\p$ of $R$. 
The main result of this section is the following.

\begin{theorem}\label{1}
Let $R$ be a commutative ring, and let
$
\cdots\xrightarrow{\alpha}L\xrightarrow{\beta}L\xrightarrow{\alpha}L\xrightarrow{\beta}\cdots
$
be an exact sequence of $R$-modules.
Let $x$ be an $L$-regular element such that $\beta(L)\subseteq\alpha(L)+xL=:Z$.
Let $M$ be an $R$-module with $Z\subseteq M\subseteq L$, and set $N=\beta(M)+xZ$.
\begin{enumerate}[\rm(1)]
\item
The sequence
$$
0 \to Z \xrightarrow{\left(\begin{smallmatrix}\theta\\
\eta\end{smallmatrix}\right)} M\oplus Z \xrightarrow{\left(\begin{smallmatrix}\beta&-x\end{smallmatrix}\right)} N \to 0
$$
is exact, where $\theta$ is the inclusion map and $\eta(\alpha(s)+xt)=\beta(t)$ for $s,t\in L$.
\item
Suppose that $R$ contains a field, $Z,M$ are finitely generated and $\beta$ is nilpotent.
Then $M$ degenerates to $N$.
\item
Assume $R$ is noetherian.
If $L,M$ are Cohen-Macaulay, then so are $Z,N$.
\end{enumerate}
\end{theorem}

\begin{proof}
(1) As the image of $\beta$ is contained in $Z$, so is that of $\eta$.
Take any element $z\in Z$.
Then $z=\alpha(s)+xt$ for some $s,t\in L$, and we have $\beta(z)=\beta(\alpha(s)+xt)=x\beta(t)$.
Since $x$ is $L$-regular and $\beta(t)\in\beta(L)\subseteq Z$, the assignment $z\mapsto \beta(t)$ gives a map $\eta:Z\to Z$, and
\begin{equation}\label{5}
\beta(z)=x\eta(z)\text{ for all }z\in Z.
\end{equation}
As $\theta$ is injective, so is the map $\left(\begin{smallmatrix}\theta\\
\eta\end{smallmatrix}\right)$, whose image is contained in the kernel of the map $\left(\begin{smallmatrix}\beta&-x\end{smallmatrix}\right)$ because $\left(\begin{smallmatrix}\beta&-x\end{smallmatrix}\right)\left(\begin{smallmatrix}\theta\\
\eta\end{smallmatrix}\right)(\alpha(s)+xt)=\left(\begin{smallmatrix}\beta&-x\end{smallmatrix}\right)\left(\begin{smallmatrix}\alpha(s)+xt\\
\beta(t)
\end{smallmatrix}\right)=\beta(xt)-x\beta(t)=0$.
The definition of $N$ shows that $\left(\begin{smallmatrix}\beta&-x\end{smallmatrix}\right)$ is surjective.

It remains to show that the kernel of $\left(\begin{smallmatrix}\beta&-x\end{smallmatrix}\right)$ is contained in the image of $\left(\begin{smallmatrix}\theta\\\eta\end{smallmatrix}\right)$.
Let $\left(\begin{smallmatrix}m\\z\end{smallmatrix}\right)\in M\oplus Z$ such that $\left(\begin{smallmatrix}\beta&-x\end{smallmatrix}\right)\left(\begin{smallmatrix}m\\z\end{smallmatrix}\right)=0$.
Lemma \ref{7} implies that there exist $s,t\in L$ such that $\left(\begin{smallmatrix}m\\z\end{smallmatrix}\right)=\left(\begin{smallmatrix}\alpha&x\\
0&\beta\end{smallmatrix}\right)\left(\begin{smallmatrix}s\\t\end{smallmatrix}\right)=\left(\begin{smallmatrix}\alpha(s)+xt\\\beta(t)\end{smallmatrix}\right)=\left(\begin{smallmatrix}\theta\\\eta\end{smallmatrix}\right)(\alpha(s)+xt)$.
Thus we are done.

(2) The module $N$ is finitely generated, and $\beta^n=0$ for some $n>0$.
It follows by \eqref{5} that $0=\beta^n(z)=(x\eta)^n(z)=x^n\eta^n(z)$ for all $z\in Z$.
As $x^n$ is $L$-regular, we get $\eta^n(z)=0$, that is, $\eta$ is nilpotent.
We see from (1) and Remark \ref{remark of degenerations 1}(1) that $M$ degenerates to $N$.

(3) Set $W=\beta(L)+xL$.
Combining (1) and Lemma \ref{7}, we get a commutative diagram
$$
\begin{CD}
0 @>>> Z @>>> L\oplus L @>{\left(\begin{smallmatrix}\beta&-x\end{smallmatrix}\right)}>> L @>>> L/W @>>> 0\\
@. @| @A{\left(\begin{smallmatrix}i&0\\
0&j\end{smallmatrix}\right)}AA @|  \\
0 @>>> Z @>{\left(\begin{smallmatrix}\theta\\
\eta\end{smallmatrix}\right)}>> M\oplus Z @>{\left(\begin{smallmatrix}\beta&-x\end{smallmatrix}\right)}>> L @>>> L/N @>>> 0
\end{CD}
$$
with exact rows, where $i,j$ are inclusion maps.
The snake lemma gives an exact sequence
\begin{equation}\label{6}
0 \to L/M\oplus L/Z \to L/N \to L/W \to 0.
\end{equation}

From now on, assuming that $R$ is noetherian, we show that if $L,M$ are Cohen-Macaulay, then so are $Z,N$.
For this, we may assume that $R$ is a local ring of Krull dimension $d$.
The first exact sequence in Lemma \ref{7} shows that $Z,W$ are Cohen-Macaulay.
The depth lemma implies that the $R$-modules $L/M,L/Z,L/W$ have depth at least $d-1$.
It follows from \eqref{6} that $L/N$ also has depth at least $d-1$, which implies that $N$ is Cohen-Macaulay.
\end{proof}

As a direct corollary of the theorem, we obtain the following result.

\begin{corollary}\label{3}
Let $R$ be a commutative noetherian ring containing a field, and let $L$ be a finitely generated $R$-module.
Let $\alpha$ be an endomorphism of $L$ with $\Im \alpha=\ker\alpha$, and let $x\in R$ be an $L$-regular element.
Then the following statements hold.
\begin{enumerate}[\rm(1)]
\item
An $R$-module $M$ with $\alpha(L)+xL\subseteq M\subseteq L$ degenerates to $N:=\alpha(M)+x^2L$.
\item
If $L,M$ are Cohen-Macaulay, then so is $N$.
\end{enumerate}
\end{corollary}

\begin{proof}
Note $\alpha^2=0$.
Letting $\alpha=\beta$ in Theorem \ref{1} shows that $M$ degenerates to $\alpha(M)+x\alpha(L)+x^2L$, and if the former is Cohen-Macaulay, then so is the latter.
As $xL\subseteq M$, we have $x\alpha(L)=\alpha(xL)\subseteq\alpha(M)$.
Hence $\alpha(M)+x\alpha(L)+x^2L=N$.
\end{proof}


From now on, we give applications of Corollary \ref{3}.

\begin{corollary}\label{8}
Let $R$ be a commutative noetherian ring containing a field, and let $\cdots \xrightarrow{\alpha} L \xrightarrow{\beta}L \xrightarrow{\alpha} L \xrightarrow{\beta} \cdots$ be an exact sequence of $R$-modules.
Let $x$ be an $L$-regular element, and set $Z=\alpha(L)+xL$ and $W=\beta(L)+xL$.
Let $M, N$ be $R$-submodules of $L$ containing $Z, W$ respectively.
Then $M\oplus N$ degenerates to $K:=(\alpha(N)+x^2L)\oplus(\beta(M)+x^2L)$. 
If $L, M$ and $N$ are Cohen-Macaulay, then so is $K$.
\end{corollary}

\begin{proof}
The endomorphism $\gamma=\left(\begin{smallmatrix}0&\alpha\\\beta&0\end{smallmatrix}\right):L^{\oplus2}\to L^{\oplus2}$ is such that $\Im \gamma=\ker\gamma$.
Putting $X=\{\left(\begin{smallmatrix}m\\n\end{smallmatrix}\right)\mid m\in M,\,n\in N\}$, we have $\gamma(L^{\oplus2})+xL^{\oplus2}\subseteq X\subseteq L^{\oplus2}$.
Note that $X\cong M\oplus N$ and $\gamma(X)+x^2L^{\oplus2}\cong K$.
Now the assertion follows from Corollary \ref{3}.
\end{proof}

\begin{remark}
Let $R$ be a hypersurface local ring.
Let $X$ be a Cohen-Macaulay $R$-module.
Then $X$ has a free resolution of the form $\cdots\xrightarrow{\alpha}F\xrightarrow{\beta}F\xrightarrow{\alpha}F\to X \to 0$, and one can apply Corollary \ref{8} to $L=F$.
\end{remark}

\begin{corollary}\label{10}
Let $R$ be a commutative noetherian ring containing a field and $L$ a finitely generated $R$-module.
Let $\alpha$ be an endomorphism of $L$ with $\Im \alpha=\ker\alpha$, and let $x\in R$ be an $L$-regular element.
Then $\alpha(L)+x^j L$ degenerates to $\alpha(L)+x^{2i-j}L$ for all $i\ge j\ge0$. 
If $L$ is Cohen-Macaulay, then so are $\alpha(L)+x^jL$ and $\alpha(L)+x^{2i-j}L$.
\end{corollary}

\begin{proof}
We have $\alpha(L)+x^iL\subseteq\alpha(L)+x^jL\subseteq L$.
Corollary \ref{3} implies that $\alpha(L)+x^jL$ degenerates to $\alpha(\alpha(L)+x^jL)+x^{2i}L=x^j\alpha(L)+x^{2i}L\cong\alpha(L)+x^{2i-j}L$, where the isomorphism follows from the fact that $x^j$ is $L$-regular.
Fix $h\ge0$.
Lemma \ref{7} implies that $C:=\alpha(L)+x^h L$ is isomorphic to the cokernel of $\left(\begin{smallmatrix}\alpha&-x^h\\0&\alpha\end{smallmatrix}\right)$, and that the sequence $\cdots\xrightarrow{\left(\begin{smallmatrix}\alpha&-x^h\\0&\alpha\end{smallmatrix}\right)}L^2\xrightarrow{\left(\begin{smallmatrix}\alpha&x^h\\0&\alpha\end{smallmatrix}\right)}L^2\xrightarrow{\left(\begin{smallmatrix}\alpha&-x^h\\0&\alpha\end{smallmatrix}\right)}\cdots$ is exact.
Using this exact sequence, we easily observe that for each prime ideal $\p$ of $R$, an $L_\p$-sequence is a $C_\p$-sequence.
Thus, if $L$ is Cohen-Macaulay, then so is $C$.
\end{proof}

We present a couple of examples.
For such an element $x$ as in the second statement it is said that $[x,x]$ is an exact pair of zerodivisors.

\begin{example}\label{2}
Let $R$ be a commutative noetherian ring containing a field. \\
(1) Let $L$ be a finitely generated $R$-module having a submodule $C$ such that $L/C\cong C$.
Let $x\in R$ be an $L$-regular element.
Then, for all $i\ge j\ge0$, $C+x^jL$ degenerates to $C+x^{2i-j}L$. \\
(2) Let $x\in R$ be such that $(0:x)=(x)$, and let $y\in R$ be a non-zerodivisor.
Then for all integers $i\ge j\ge0$ $(x,y^j)$ degenerates to $(x,y^{2i-j})$, where $(x,y^j)$ and $(x,y^{2i-j})$ are ideals of $R$. 
\end{example}

\begin{proof}
The composition $\alpha:L\twoheadrightarrow L/C\cong C\hookrightarrow L$ satisfies $\Im \alpha=\ker\alpha$.
The first assertion is shown by Corollary \ref{10}.
The second assertion follows from the first one.
\end{proof}

\begin{remark}
Example \ref{2}(2) immediately recovers the degeneration obtained in \cite[Remark 2.5]{HY13} and the `if' part of Theorem \ref{odd dim}. 
\end{remark}

\section{Remarks on degenerations over hypersurface rings}\label{Knorrer}

In this section, we consider extending Theorem \ref{a} to higher dimension.
Let $k$ be an algebraically closed field of characteristic not two and let $R=S/(f)$ be a hypersurface, where $S = k[\![x_0, x_1, \cdots , x_n]\!]$ is a formal power series ring with maximal ideal $\m _S = (x_0, x_1, \cdots , x_n )$ and $f \in \m _S$. 
We define $\Rs = S[\![u ]\!] /(f + u^2)$ and $\Rss = S[\![u, v]\!]/(f + u^2 + v^2)$.
In what follows, for a certain hypersurface $R$, we investigate the connection among the degenerations of Cohen-Macaulay modules over $R$, $\Rs$ and $\Rss$.

\begin{remark}\label{syzygy}
Let $R = S[\![y]\!]/(y^2 + f )$ with $f \in \m _S$.\\
(1) There is a one-to-one correspondence between
\begin{itemize}
\item
the isomorphism classes of Cohen-Macaulay $R$-modules, and
\item
the equivalence classes of square matrices $\varphi$ with entries in $S$ such that $\varphi ^2 = -f$.
\end{itemize}
This follows by considering matrix representations over $S$.\\
(2) Let $M$ be a Cohen-Macaulay $R$-module, and let $\Omega M$ be the first syzygy of $M$ with respect to the minimal $R$-free resolution of $M$.
As remarked in \cite[Lemma 12.2]{Y}, the matrix representation of $\Omega M$ over $S$ is $-\mu$, where $\mu$ is that of $M$.\\
\end{remark}

Let $R$ be as in the above remark.
Let $\mu$ be a square matrix with entries in $S$ such that $\mu ^2 = -f$ and $M$ a Cohen-Macaulay $R$-module which corresponds to $\mu$. 
Consider the matrices $\mu ^{\sharp} = \left( \begin{smallmatrix} \mu & u \\ -u & -\mu \end{smallmatrix}\right)$ and $\mu ^{\sharp \sharp}= \left( \begin{smallmatrix} \mu & \zeta \\ -\eta & -\mu \end{smallmatrix}\right)$, where $\zeta = u + \sqrt{-1}v$ and $\eta =  u - \sqrt{-1}v$. 
By Remark \ref{syzygy}(1) we have $(\mu ^{\sharp} )^2 = -(f + u^2)$ (resp. $(\mu ^{\sharp \sharp} )^2 = -(f + u^2 + v^2)$), and $\mu ^{\sharp}$ (resp. $\mu ^{\sharp \sharp}$) corresponds to a Cohen-Macaulay module over $\Rs$ (resp. $\Rss$), which we denote by $M^{\sharp}$ (resp. $M^{\sharp \sharp}$).
Note that $(M^{\sharp} )^{\sharp}$ is not always isomorphic to $M^{\sharp \sharp}$. 

To show our main result, we state and prove a lemma.

\begin{lemma}\label{regular element}
Let $R$ be a commutative noetherian ring and $M,N$ finitely generated $R$-modules. 
If $M$ degenerates to $N$, then $M/xM$ also degenerates to $N/xN$ for an $N$-regular element $x$. 
\end{lemma}

\begin{proof}
Suppose that $M$ degenerates to $N$ along $V$. 
As mentioned in Remark \ref{remark of degenerations 1}(1), we have an exact sequence
$
0 \to Z \to M \oplus Z \to N \to 0,   
$
where the endomorphism $g$ of $Z$ is nilpotent. 
Since $x$ is $N$-regular, we also have
$
0 \to Z/xZ \to M/xM \oplus Z/xZ \to N/xN \to 0.
$
The endomorphism $g \otimes R/xR$ is also nilpotent, so that $M/xM$ degenerates to $N/xN$ along $V$.  
\end{proof}

The following is the main result of this section.

\begin{proposition}\label{Knorrer of mr}
Let $R = S[\![y]\!]/(y^2 + f )$ be a hypersurface, and let $M,N$ be Cohen-Macaulay $R$-modules.
Consider the following six conditions.

\smallskip

\noindent
\begin{tabular}{lll}
{\rm(1)} $M$ degenerates to $N$.
&{\rm(2)} $M^{\sharp \sharp}$ degenerates to $N^{\sharp \sharp}$.
&{\rm(3)} $M^{\sharp}$ degenerates to $N^{\sharp}$.
\\
{\rm(1')} $M$ degenerates to $\Omega N$.\qquad
&{\rm(2')} $M^{\sharp \sharp}$ degenerates to $(\Omega N )^{\sharp \sharp}$.
&{\rm(4)} $M \oplus \Omega M $ degenerates to $N \oplus \Omega N$.
\end{tabular}

\smallskip

\noindent
Then the implications {\rm(1)} $\Rightarrow$ {\rm(2)} $\Rightarrow$ {\rm(3)} $\Rightarrow$ {\rm(4)} and {\rm(1')} $\Rightarrow$ {\rm(2')} $\Rightarrow$ {\rm(3)} $\Rightarrow$ {\rm(4)} hold true.
\end{proposition}

\begin{proof}
$(1) \Rightarrow (2)$: 
Assume that $M$ degenerates to $N$ along $V$.
Let $\mu$ and $\nu$ be the matrix representations of $M$ and $N$ over $S$. 
We have the matrix representation $\xi$ over $S \otimes _k V$ such that $\xi \otimes_V V_t \cong \mu \otimes _k V_t$ and $\xi \otimes _V V/tV \cong \nu$. 
Now consider the matrix
$
\xi ^{\sharp \sharp} = \left( \begin{smallmatrix} \xi & \zeta \\ -\eta & -\xi \end{smallmatrix}\right)
$, where $\zeta = u + \sqrt{-1}v$ and $\eta =  u - \sqrt{-1}v$.
Then one can show that $\xi ^{\sharp \sharp} \otimes_V V_t \cong \mu^{\sharp \sharp}$ and $\xi ^{\sharp \sharp} \otimes_V V/tV \cong \nu^{\sharp \sharp}$. 
In fact, since $\alpha^{-1} (\xi \otimes_V V_t ) \alpha = \mu$ and $\beta^{-1} (\xi \otimes _V V/tV) \beta = \nu$ respectively for some suitable invertible matrices $\alpha$ and $\beta$, we have
$$
\begin{array}{ll}
\left(\begin{smallmatrix}
\alpha^{-1} & 0 \\
0 & \alpha^{-1}
\end{smallmatrix}\right)
( \xi^{\sharp \sharp} \otimes_V V_t )
\left( \begin{smallmatrix}
\alpha & 0 \\
0 & \alpha
\end{smallmatrix}
\right) 
&= 
\left( \begin{smallmatrix}\alpha^{-1}( \xi  \otimes_V V_t  )\alpha&\alpha^{-1} (\zeta   \otimes_V V_t )\alpha \\\alpha^{-1} (-\eta \otimes_V V_t ) \alpha &\alpha^{-1}( -\xi  \otimes_V V_t )\alpha \end{smallmatrix}\right)= \left( \begin{smallmatrix}\mu &\zeta \\ -\eta  & -\mu \end{smallmatrix}\right)= \mu ^{\sharp \sharp}.
\end{array}
$$ 
Similarly, 
$$
\begin{array}{ll}
\left( \begin{smallmatrix}
\beta^{-1} & 0 \\
0 & \beta^{-1}
\end{smallmatrix}
\right)
( \xi^{\sharp \sharp} \otimes_V V/tV )
\left( \begin{smallmatrix}
\beta & 0 \\
0 & \beta
\end{smallmatrix}
\right) 
&= 
\left( \begin{smallmatrix}\beta ^{-1}( \xi  \otimes_V V/tV  )\beta &\beta ^{-1} (\zeta \otimes_V V/tV )\beta \\ \beta ^{-1} (-\eta \otimes_V V/tV ) \beta &\beta^{-1}( -\xi  \otimes_V V/tV )\beta \end{smallmatrix} \right)=\left( \begin{smallmatrix}\nu &\zeta \\ -\eta  & -\nu \end{smallmatrix}\right)=\nu ^{\sharp \sharp}.
\end{array}
$$ 
According to Corollary \ref{degeneration via MR}, we conclude that $M^{\sharp \sharp}$ degenerates to $N^{\sharp \sharp}$.   

$(1') \Rightarrow (2')$: The similar argument to the proof of $(1)\Rightarrow(2)$ shows this. 

$(2) \Rightarrow (3)$: By Lemma \ref{regular element}, we see that $M^{\sharp \sharp}/ v M^{\sharp \sharp}$ degenerates to $N^{\sharp \sharp}/ v N^{\sharp \sharp}$. 
Clearly, $M^{\sharp \sharp}/ v M^{\sharp \sharp} \cong M^{\sharp}$ and $N^{\sharp \sharp}/ v N^{\sharp \sharp} \cong N^{\sharp}$. 
Thus we obtain the assertion. 

$(2') \Rightarrow (3)$: We also have the degeneration from $M^{\sharp \sharp}/ v M^{\sharp \sharp} \cong M^{\sharp}$ to $(\Omega N )^{\sharp \sharp} / v (\Omega N )^{\sharp \sharp}$ by Lemma \ref{regular element}. 
We note that $(\Omega N )^{\sharp \sharp} / v (\Omega N )^{\sharp \sharp} \cong N^{\sharp}$. 
Actually, for the matrix representation $\left ( \begin{smallmatrix} -\nu & u \\ -u & \nu \end{smallmatrix}\right)$ of $(\Omega N )^{\sharp \sharp} / v (\Omega N )^{\sharp \sharp}$, we have   
$
\left( \begin{smallmatrix}  0& -1 \\ 1 & 0 \end{smallmatrix}\right) \left( \begin{smallmatrix} -\nu & u \\ -u & \nu \end{smallmatrix}\right) \left( \begin{smallmatrix} 0 & 1 \\ -1 & 0 \end{smallmatrix}\right) = \left( \begin{smallmatrix} \nu & u \\ -u & -\nu \end{smallmatrix}\right). 
$
Thus, the implication holds. 

$(3) \Rightarrow (4)$: It follows from Lemma \ref{regular element} since $M^{\sharp}/ u M^{\sharp} \cong M \oplus \Omega M$ and $N^{\sharp}/ u N^{\sharp} \cong N \oplus \Omega N$.
\end{proof}

\begin{remark}
In Proposition \ref{Knorrer of mr}, (4) does not necessarily imply (1) or (1').
Let $R = k[\![x, y, z]\!]/(x^3 + y^2 + z^2)$ and $\p$, $\q$, $\r$ be ideals $(x, y)$, $(x, y^2)$, $(x, y^3)$ of $R$ respectively. 
Then we have the exact sequence $0 \to \q \to \p \oplus \r \to \q  \to 0$, so that $\p \oplus \r$ degenerates to $\q \oplus \q$. 
Notice that $\r \cong \Omega \p$ and $\q \cong \Omega \q$. 
However $\q$ never degenerates to $\p$ or $\q$. 
See \cite[Theorem 3.1]{HY13} for the details. 
\end{remark}

Recall that Kn\"{o}rrer's periodicity theorem \cite[Theorem 12.10]{Y04} gives an equivalence
$
\Phi :\SCM (R) \to \SCM (\Rss )
$
of triangulated categories.
We call this functor a {\em Kn\"{o}rrer's periodicity functor}.
Using Proposition \ref{Knorrer of mr}, we generalize the `if' part of Theorem \ref{odd dim}. 

\begin{example}\label{14}
Consider the hypersurface singularity of type $(A_\infty)$ with odd dimension:
$$
k[\![x_0,x_1,\dots,x_n]\!]/(x_1 ^2+x_2 ^2 +\cdots+x_n ^2 ).
$$
Let $M(h)$ be the image of $(x_0,z^h)$ by the composition of Kn\"{o}rrer's periodicity functors
$
\SCM (k[\![x_0,x_1]\!]/(x_1^2))\to\SCM (k[\![x_0,x_1,x_2,x_3]\!]/(x_1^2+x_2 ^2 + x_3 ^2))\to\cdots\to\SCM (k[\![x_0,x_1,\dots,x_n]\!]/(x_1 ^2+x_2 ^2 +\cdots+x_n ^2 ))$.
Iterated application of Proposition \ref{Knorrer of mr} shows that $M(j)$ degenerates to $M(2i-j)$.
\end{example}

\begin{proposition}
Let $R=S/(f)$ with $f \in \m _S$ and let $\alpha:S^n\to S^n$ be an endomorphism such that $(\alpha,\alpha)$ is a matrix factorization of $f$ over $S$.
Let $z\in R$ be a non-zerodivisor, and consider the map $\left(\begin{smallmatrix}\alpha&z^h\end{smallmatrix}\right):R^n\oplus R^n\to R^n$ for each $h\ge0$. 
Set $\zeta = u + \sqrt{-1}v$ and $\eta =  u - \sqrt{-1}v$. 
Then the following hold.
\begin{enumerate}[\rm(a)]
\item 
The images $\Im \left(\begin{smallmatrix}\alpha&z^h\end{smallmatrix}\right)$ and $\Im \left(\begin{smallmatrix}\alpha&\zeta &z^h&0\\ \eta&-\alpha&0&z^h\end{smallmatrix}\right)$ are Cohen-Macaulay modules over $R$ and $\Rss$, respectively.
\item
The functor $\Phi$ sends $\Im \left(\begin{smallmatrix}\alpha&z^h\end{smallmatrix}\right)$ to $\Im \left(\begin{smallmatrix}\alpha&\zeta &z^h&0\\ \eta&-\alpha&0&z^h\end{smallmatrix}\right)$.
\item
The $R$-module $\Im \left(\begin{smallmatrix}\alpha&z^h\end{smallmatrix}\right)$ is nonfree and indecomposable if and only if so is the $\Rss$-module $\Im \left(\begin{smallmatrix}\alpha&\zeta &z^h&0\\ \eta &-\alpha&0&z^h\end{smallmatrix}\right)$.
\item
For all $i\ge j\ge0$ one can show that $\Im {\left(\begin{smallmatrix}\alpha&z^j\end{smallmatrix}\right)}$ degenerates to $ \Im {\left(\begin{smallmatrix}\alpha&z^{2i-j}\end{smallmatrix}\right)}$ and $\Im {\left(\begin{smallmatrix} \alpha&\zeta &z^j&0\\ \eta &-\alpha&0&z^j \end{smallmatrix}\right)}$ degenerates to $\Im {\left(\begin{smallmatrix}\alpha&\zeta &z^{2i-j}&0\\ \eta &-\alpha&0&z^{2i-j}\end{smallmatrix}\right)}$ as Cohen-Macaulay modules over $R$, $\Rss$ respectively.
\end{enumerate}
\end{proposition}

\begin{proof}
We regard $\alpha$ and $\alpha'=\left(\begin{smallmatrix}
\alpha&\zeta \\
\eta &-\alpha
\end{smallmatrix}\right)$ as the endomorphisms of $R^n$ and ${\Rss}^{2n}$, respectively.
Note that $\Im \alpha=\ker\alpha$ and $\Im \alpha'=\ker\alpha'$.
Assertions (a) and (d) follow from Corollary \ref{10}.
By \cite[(12.8.1)]{Y}, the cokernel of $\left(\begin{smallmatrix}\alpha&-z^h\\0&\alpha\end{smallmatrix}\right)$ is sent by $\Phi$ to that of $\left(\begin{smallmatrix}
\alpha&-z^h&\zeta &0\\
0&\alpha&0&\zeta \\
\eta &0&-\alpha&-z^h\\
0&\eta &0&-\alpha
\end{smallmatrix}\right)$, which is transformed into $\left(\begin{smallmatrix}
\alpha&\zeta &-z^h&0\\
\eta &-\alpha&0&-z^h\\
0&0&\alpha&\zeta \\
0&0&\eta&-\alpha
\end{smallmatrix}\right)$ by elementary row and column operations.
By Lemma \ref{7}, the cokernel of the latter matrix is the image of $\left(\begin{smallmatrix}
\alpha&\zeta &z^h&0\\
\eta &-\alpha&0&z^h
\end{smallmatrix}\right)$.
Thus (b) follows.
As $\Phi$ is an equivalence of additive categories, (c) follows from (b).
\end{proof}

\begin{remark}
(1) Since, for each $h>0$, the ideal $(x_0,z^h)$ of the ring $k[\![x_0,z]\!]/(x_0^2)$ is nonfree and indecomposable and coincides with $\Im \left(\begin{smallmatrix}x_0&z^h\end{smallmatrix}\right)$,  Example \ref{14} can also be obtained from the above proposition.
\\
(2) Let $R$ be an algebra over a field.
Let $0\to L\to M\to N\to0$ be an exact sequence of finitely generated $R$-modules.
Then $M$ degenerates to $L\oplus N$ by Remark \ref{remark of degenerations 1}(2).
Such a degeneration is called a {\em degeneration by an extension} in \cite[Definition 2.4]{HY13}.
Thus, it is important to investigate the existence of degenerations that cannot be obtained by extensions.
The degeneration in Example \ref{14} cannot be obtained by (iterated) extensions, because $M(2i-j)$ is an indecomposable $R$-module.
\end{remark}

\end{document}